\documentclass{amsart}

\usepackage{amssymb,verbatim,amsmath}
\usepackage{graphicx}
\usepackage{psfrag,color}
\usepackage{pxfonts}

\newtheorem{theorem}{Theorem}[section]
\newtheorem{lemma}[theorem]{Lemma}
\newtheorem{proposition}[theorem]{Proposition}

\theoremstyle{definition}

\theoremstyle{remark}
\newtheorem{remark}[theorem]{Remark}

\numberwithin{equation}{section}

\newcommand{\abs}[1]{\left|#1\right|}
\newcommand{\C}{\mathbb{C}}
\newcommand{\diff}{\operatorname{d}}

\renewcommand{\H}{\mathbb{H}}
\newcommand{\id}{\operatorname{id}}

\newcommand{\N}{\mathbb{N}}
\newcommand{\Nil}{\operatorname{Nil}}
\newcommand{\nor}{\operatorname{norm}}

\newcommand{\PSL}{\operatorname{PSL}}
\newcommand{\R}{\mathbb{R}}
\renewcommand{\S}{\mathbb{S}}

\newcommand{\T}{\operatorname{T}}

\newcommand{\vol}{\operatorname{area}}

\renewcommand{\phi}{\varphi}
\renewcommand{\theta}{\vartheta}

\newcommand{\la}{\langle}
\newcommand{\ra}{\rangle}

\begin{document}

\title{}
\title
{Constant mean curvature $k$-noids in homogeneous manifolds}
\date{\today}
\author[Plehnert]{Julia Plehnert}
\address{Georg-August-Universit\"at G\"ottingen, Research Group Discrete Differential Geometry, Faculty of Mathematics, Lotzestrasse 16-18, 37083 G\"ottingen, Germany}
\email{j.plehnert@math.uni-goettingen.de}


\begin{abstract}
For each $k\geq2$, we construct two families of surfaces with constant mean curvature $H$ for $H\in[0,1/2]$ in $\Sigma(\kappa)\times\R$ where $\kappa+4H^2\leq0$. The surfaces are invariant under $2\pi/k$-rotations about a vertical fiber of $\Sigma(\kappa)\times\R$, have genus zero, and a finite number of ends. The first family generalizes the notion of $k$-noids: It has $k$ ends, one horizontal and $k$ vertical symmetry planes. The second family is less symmetric and has two types of ends. Each surface arises as the conjugate (sister) surface of a minimal graph in a homogeneous $3$-manifold. The domain of the graph is non-convex in the second family. For $\kappa=-1$ the surfaces with constant mean curvature $H$ arise from a minimal surface in $\widetilde{\PSL}_2(\R)$ for $H\in(0,1/2)$ and in $\Nil$ for $H=1/2$. For $H=0$, the conjugate surfaces are both minimal in a product space.
\end{abstract}

\maketitle

\section{Introduction}

Recently various mathematicians constructed minimal (\cite{MR2012,pyo2011,younes2010}) and constant mean curvature (\cite{MT2011,GK2010}) surfaces in three dimensional homogeneous manifolds with four dimensional isometry group via minimal (vertical) graphs above convex domains $\Omega$. The existence of minimal surfaces follows since in those manifolds the preimage $\pi^{-1}(\Omega)$ under the Riemannian fibration $\pi$ is a mean convex domain. In this paper we construct surfaces with constant mean curvature $H\in[0,1/2]$, which arise from minimal graphs, to some extent above non-convex domains. 

The paper begins with the setup of sister surfaces in homogeneous manifolds in Section \ref{S:Sisterhom}. In Section \ref{S:refsurfaces} we define a number of reference surfaces, which we use as barriers. First of all, we summarize a classification of ruled minimal surfaces in homogeneous manifolds. Followed by a subsection on graphs where we prove the existence of a Scherk type minimal surface in homogeneous manifolds, which generalizes known surfaces. The section ends with the construction of  $k$-noids with constant mean curvature $H\in[0,1/2]$, we use the minimal sister as a barrier. In the final section we construct the family of the less symmetric $2k$-noids. The main step is the construction of an appropriate mean convex domain in order to solve a Plateau problem, see Section \ref{SS:Plateau}.

\section{Sister surfaces in homogeneous $3$-manifolds}\label{S:Sisterhom}

We construct constant mean curvature (cmc) surfaces in simply connected homogeneous $3$-manifolds with an at least $4$-dimensional isometry group. Such a manifold is a Riemannian fibration $\pi\colon E\to\Sigma(\kappa)$ with geodesic fibers and constant bundle curvature $\tau$, where $\Sigma(\kappa)$ is a two-dimensional space form. We write $E(\kappa,\tau)$ for those spaces, see \cite{daniel2007}. 

An interpretation of the bundle curvature is the vertical distance of a horizontal lift of a closed curve, which was proven in \cite{plehnert2013}.

\begin{lemma}[Vertical distances]\label{l:vertdistances}
Let $\gamma$ be a closed Jordan curve in the base manifold $\Sigma(\kappa)$ of a Riemannian fibration with constant bundle curvature $\tau$ and geodesic fibers. Let $\Delta$ be the bounded domain defined by $\partial\Delta=\gamma$, we have
\[
d(\tilde{\gamma}(0),\tilde{\gamma}(l))=2\tau\vol(\Delta),
\]
where $\tilde{\gamma}$ is the horizontal lift with $\pi(\tilde{\gamma}(0))=\pi(\tilde{\gamma}(l))$, $\vol$ is the oriented volume and $d(p,q)$ denotes the signed vertical distance, which is positive if $\overline{pq}$ is in fiber-direction $\xi$.
\end{lemma}

One way to construct cmc surfaces is the conjugate Plateau construction, see \cite{karcher2005} for an introduction in the case of space forms. This approach uses the Lawson correspondence between isometric surfaces in space forms \cite{lawson1970}. Recently Daniel generalized the correspondence to homogeneous $3$-manifolds. Since the construction uses reflection about vertical and horizontal planes, we consider only a special case of Daniel's correspondence:

\begin{theorem}[{\cite[Theorem 5.2]{daniel2007}}] \label{t:correspondence} There exists an isometric correspondence 
between an MC $H$-surface $\tilde{M}$ in $\Sigma(\kappa)\times\R=E(\kappa,0)$ and a minimal surface $M$ in $E(\kappa+4H^2,H)$. Their shape operators are related by \begin{equation}\label{e:shapeoperators}\tilde{S}=JS+H\id,\end{equation} where $J$ denotes the $\pi/2$ rotation on the tangent bundle of a surface. Moreover, the normal and tangential projections of the vertical vector fields $\xi$ and $\tilde{\xi}$ are related by
\begin{equation}
\la\tilde{\xi},\tilde{\nu}\ra=\la\xi,\nu\ra,\qquad J\diff f^{-1}(T)=\diff \tilde{f}^{-1}(\tilde{T}),
\end{equation}
where $f$ and $\tilde{f}$ denote the parametrizations of $M$ and $\tilde{M}$ respectively, $\nu$ and $\tilde{\nu}$ their unit normals, and $T$, $\tilde{T}$ the projections of the vertical vector fields on $\T M$ and $\T\tilde{M}$.
\end{theorem}

We call the isometric surfaces $M$ and $\tilde{M}$ sister surfaces, or sisters in short.

The idea of the conjugate Plateau construction is to solve a Plateau problem for a polygon, which consists of horizontal and vertical geodesics. Then the sister surface is bounded by a piecewise smooth curve contained in totally geodesic vertical/horizontal planes and the surface conormal is perpendicular to those planes, the so-called mirror planes, see \cite{MT2011}. Under certain assumptions Schwarz reflection about the horizontal and vertical mirror planes extends the surfaces smoothly without branch points.

We call related curves $\tilde{c}\subset\tilde{M}$ and $c\subset M$ sister curves. One computes directly that their normal curvatures and torsions are associated as follows:
$$\tilde{k}=-t+H\quad\text{and}\quad\tilde{t}=k.$$ In \cite{plehnert2013} we defind the twist $\alpha$ of the normal along a vertical geodesic $c$ as its total rotation speed with respect to a basic vector field, and proved  that
\begin{equation}\label{E:twist}
\alpha=\int\limits_c t+H l(c)\quad\mbox{and}\quad\tilde{k}=2H-\alpha'.
\end{equation}

We are interested in surfaces with cmc $H$ in product spaces $\Sigma(\kappa)\times\R$ with $\kappa\leq0$. The behaviour of the surfaces depends on $(H,\kappa)$. Let us distinguish two cases:
\begin{itemize}
\item $\kappa=0$: There is an isometric correspondence between surfaces with cmc $H$ in $\R^3$ and minimal surfaces in $3$-dimensional space forms with curvature $H$. Since the hyperbolic $3$-manifold is not a Riemannian fibration this case is not covered by the realtion above and we do not treat it here. For $H=0$ we get two conjugate minimal surfaces in $\R^3$. If $H>0$ the minimal surface is constructed in a Berger sphere, since the base manifold of this fibration is compact, we do not consider this case. 

\item $\kappa<0$: All product spaces have the same isometry group, so we consider $\H^2\times\R$, hence $\kappa=-1$. There are different cases: For $H=0$ the relation describes two minimal surfaces in $\H^2\times\R$, for $H\in(0,1/2)$ the surfaces arise from minimal surfaces in $E(4H^2-1,H)$, which has the same isometry group as $\tilde{\PSL}_2(\R)$. The third class of surfaces has constant mean curvature $1/2$ and its sister surface is minimal in $\Nil_3$. We do not consider the case $H>1/2$, since this would lead to the compact base manifold as above.
\end{itemize}

\section{Reference surfaces}\label{S:refsurfaces}

\subsection{Ruled surfaces in homogeneous $3$-manifolds}\label{SS:umbrella}
In \cite{GK2009} Gro\ss{}e-Brauckmann and Kusner discussed ruled minimal surfaces in homogeneous manifolds and their sisters systematically. Since we need some of them as barriers in our surface construction, we present a short outline of their work.

In $\R^3$ a ruled surface is defined for an arc-length parametrized curve $c\colon I\to\R^3$ and an unit vector field $v(t)$ along $c$ with $v(t)\perp c'(t)$ as the mapping $$f\colon\R\times I\to \R^3,\quad f(s,t)\coloneqq c(t)+s v(t).$$ The curve $c$ is called directrix; the rulings $\gamma(s)\coloneqq f(s,t_0)$ are asymptote lines. The classic examples of ruled surfaces are: cylinder, cone and hyperbolic paraboloid (doubly ruled).

The helicoid $f(s,t)=(s\cos t,s\sin t, ht)$, $h\in\R\cup\{\pm\infty\}$ is also a ruled surface. Its axis is a vertical geodesic $c(t)\coloneqq(0,0,ht)$ and it has horizontal geodesics as rulings $\gamma(s)\coloneqq (s\cos t_0,s\sin t_0,ht_0)$. The pitch is given by the parameter $h$, it controls the constant rotation-speed. For $h=0$ it is a horizontal plane and for $h=\pm\infty$ it is a vertical plane. We claim, it is minimal, because the helicoid is invariant under $\pi$-rotation about its rulings. Let $\tilde{\gamma}$ be a curve, which is perpendicular to a ruling $\gamma$. The normal curvature $\kappa_{\nor}(\tilde{\gamma})$ changes sign under rotation, therefore $\kappa_{\nor}(\tilde{\gamma})=0$, i.e. $\tilde{\gamma}$ is an asymptotic direction and perpendicular to $\gamma$. With the Euler-curvature-formula $g(Sv_\alpha,v_\alpha)=\kappa_1\cos^2\alpha+\kappa_2\sin^2\alpha$ we get $\kappa_1=-\kappa_2$, since $g(Sv_{\alpha_i},v_{\alpha_i})=0,\,i=1,2$ where $\alpha_1=\alpha_2-\pi/2$. Actually each complete ruled minimal surface is either the plane or the helicoid.

In $E(\kappa,\tau)$ each $\pi$-rotation about horizontal and vertical geodesics is an isometry. Therefore, surfaces which are invariant under $\pi$-rotations about those geodesics are minimal. Hence, we consider surfaces foliated by geodesics.

\begin{enumerate}
\item Vertical planes

A vertical plane is defined as the preimage $\pi^{-1}(c)$ of a geodesic $c\subset\Sigma$. Vertical planes are minimal, since the horizontal lift of $c$ is a geodesic. Therefore, the surface is foliated by geodesics. Moreover, a $\pi$-rotation about each geodesic leaves the plane invariant. In a product space we have for example $\{c\}\times \R$  for a geodesic $c\in\Sigma(\kappa)$. In $E(4,1)=\S^3$ a vertical plane is a Clifford torus.

\item Horizontal umbrellas

Horizontal umbrellas correspond to horizontal planes. They are defined by the exponential map of the horizontal tangent subspace of $\T_p E(\kappa,\tau)$ in a point $p\in E(\kappa,\tau)$. Therefore, each umbrella consists of all horizontal radial geodesics starting at $p$. In $\Sigma(\kappa)\times\R$ a horizontal umbrella is totally geodesic, whereas for $\tau\ne0$ the surface has non-horizontal tangent spaces except in $p$. Horizontal umbrellas are minimal. For $\kappa\leq0$ or $\tau=0$ they are sections. Each surface is of disc type for $\kappa\leq0$ and a sphere for $\kappa>0$, for example a geodesic $2$-sphere in $\S^3$.

\item Horizontal slices

We interpret a horizontal slice as a horizontal helicoid, where the axis is a horizontal geodesic $c$ and the rulings are the horizontal geodesics, which are perpendicular to $c$. Horizontal slices are minimal. Topological it is a disc for $\kappa\leq0$, a torus if $\kappa>0, \tau\ne0$, and a sphere if $\kappa>0, \tau=0$.

\item Vertical helicoids

Last but not least we consider vertical helicoids $M(s)$. It is family of minimal surfaces, where the axis is a fiber of $\pi:E\to\Sigma$ and the rulings are horizontal geodesics, which rotate along the axis with constant speed $s$.
As in $\R^3$, we have special cases: The surface $M(\tau)\subset E(\kappa,\tau)$ is a vertical plane and $M(\pm\infty)$ are horizontal umbrellas.
\end{enumerate}

\subsection{Minimal surface equation for graphs}

To derive a minimal surface equation we consider $\R_{>0}\times\R^2$ with
\[
\diff s^2=\lambda^2(\diff x^2+\diff y^2)+(2Hy\lambda^2\diff x+\diff z)^2,
\]
where $\lambda=\frac{1}{\sqrt{-\kappa-4H^2} y}$. This is a model for $E(\kappa+4H^2,H)$ with $\kappa+4H^2<0$.

The Riemannian fibration $\pi\colon E(\kappa+4H^2,H)\to\Sigma(\kappa+4H^2)$ for these coordinates is given by the projection onto the first two coordinates. The vertical vector field is $\xi=\partial_{z}$. 

We generalise the notion of a graph in Riemannian fibrations with geodesic fibers. Let $\pi\colon E\to B$ be a fiber bundle over a base space $B$. Then a continuous map $s\colon B\to E$ is called section if $\pi(s(x))=x$ for all $x\in B$. Let $E$ be the Riemannian fibration $E(\kappa+4H^2,H),\,\kappa+4H^2<0$. We call the surface $\{s(x)\in E(\kappa+4H^2,H)\colon x\in \Omega\}$ a graph over $\Omega\subset \Sigma(\kappa+4H^2)$ if $s$ is transversal to the fibers.

Let $M$ be a coordinate graph $z=u(x,y)$ in $E\coloneqq E(\kappa+4H^2,H)$ endowed with the metric from above. 

Then $M$ is minimal if $u$ is a solution of
\begin{equation}\label{E:MCE}
2w(u_{xx}+u_{yy})-\left(\left(\frac{2H}{y(-4H^2-\kappa)}+u_x\right)w_x+u_y w_y\right)=0,
\end{equation}
where $w=1+(u_x/\lambda+2H\lambda y)^2+(u_y/\lambda)^2$.

Let us change the coordinates $x=r\cos s, y=r\sin s,\,r>0,\,0<s<\pi/2$ and assume that the solution $u(r,s)$ is constant along $r\mapsto u(r,s)$, i.e. $\partial_r u=0$ and therefore let $\cdot'$ denote the derivate with respect to $s$. Equation \eqref{E:MCE} is then equivalent to  
\[
2w(4H^2+\kappa)u''-w'(2H+(4H^2+\kappa)u')=0.
       \]
    
      This is equivalent to       
       \[
       \frac{(2H+(4H^2+\kappa)u')^2}{w}
       =c,\,\mbox{ for a constant } c\in\R.
      \]
   Using $w=1-\frac{4H}{4H^2+\kappa}-4H\sin^2(s)u'-(4H^2+\kappa)\sin^2(s)u'^2$ we get
   \[
   u'(s)=\frac{1}{4H^2+\kappa}\left(2H\pm\sqrt{\frac{(4H^2\cos^2 s-(4H^2+\kappa))\frac{-c}{4H^2+\kappa}}{1-\frac{-c}{4H^2+\kappa}\sin^2s}}\right),
   \]
   hence with $c=-(4H^2+\kappa)$
   \begin{equation}\label{E:Scherk}
   u(s)=\frac{1}{4H^2+\kappa}
   	\left(
   		2Hs\pm\int\limits_0^s\sqrt{
   			\frac{4H^2\cos^2 t-(4H^2+\kappa)}{1-\sin^2t}}\diff t \right)
   \end{equation}
   is a minimal section.
   
For $\kappa=-1$ and $H=0$ this surface arises also in a family of screw motion surfaces in $\H^2\times\R$ deduced in \cite{earp2008}, moreover in \cite{CR2010} it is used in the construction of a complete minimal graph to prove the existence of a harmonic diffeomorphism from $\C$ to $\H^2$.
      
      In the universal cover of $\PSL_2(\R)$ the solution was derived in \cite{younes2010} and used to prove a Jenkins-Serrin type theorem on compact domains.
   
In general Equation \eqref{E:Scherk} parametrizes a minimal graph of Scherk type in $E(\kappa+4H^2,H),\,\kappa+4H^2<0$, where we consider polar coordinates in the upper half-plane model of $\Sigma(\kappa+4H^2)$. For the right choice of signs the surface has zero boundary data for $s=0$ (positive $x$-axis) and tends to infinity for $s\to \pi/2$. Hence $M$ is a minimal graph above a domain, which is bounded by a geodesic.

By isometries of $E(\kappa+4H^2,H)$ we may define a minimal graph $M$ for any geodesic $\gamma\subset\Sigma(\kappa+4H^2)$, such that $M$ converges to infinity on $\gamma$ and has asymptotic values zero on a subset of $\partial \Sigma(\kappa+4H^2)$.

\subsection{Constant mean curvature $k$-noids in $\Sigma(\kappa)\times\R$}\label{S:knoid}

In \cite{GK2010} Gro\ss{}e-Brauckmann and Kusner describe the conjugate Plateau construction in $E(\kappa,\tau)$. They outline the construction of an one-parameter family of surfaces with constant mean curvature $H\geq0$ in $\Sigma(\kappa)\times\R$, which have $k$ ends and dihedral symmetry. The idea is to solve a compact Plateau problem of disc type in $E(\kappa+4H^2,H)$ and take a limit of Plateau solutions in order to solve an improper Plateau problem. Its sister generates the cmc surface by reflections about horizontal and vertical planes. We use the limiting minimal disc $M=M(a,k)$ as a barrier in our construction.

We prove that the minimal surface $M$ is the limit of a sequence of compact Plateau solutions $M_{(r,s)}$, which represent sections in $E(\kappa+4H^2,H)$, since the sequence is bounded. Each minimal disc $M_{(r,s)}$ is bounded by horizontal and vertical geodesics, see Figure \ref{f:knoid}. Let $\Gamma_{(r,s)}$ denote the boundary. The minimal surface $M_{(r,s)}$ is a section of the trivial line bundle $\pi\colon\Omega_r\subset E(\kappa+4H^2,H)\to\Delta_r$, where $\Omega_r\coloneqq\pi^{-1}(\Delta_r)$ is a mean convex domain, which is defined as the preimage of a triangle $\Delta_r\subset\Sigma(\kappa+4H^2)$. The triangle $\Delta_r$ is given by a hinge of lengths $a$ and $r$, enclosing an angle $\pi/k$. The parameter $a$ determines the length of the horizontal edge in the boundary of $M$, it defines the necksize in the cmc sister.

\begin{figure}[h]\begin{center}
\psfrag{a}{$a$}
\psfrag{r}{$r$}
\psfrag{pr}{$\pi$}
\psfrag{s}{$s$}
\psfrag{g}{$\gamma_r$}
\psfrag{phi}{$\pi/k$}
\includegraphics[width=5cm]{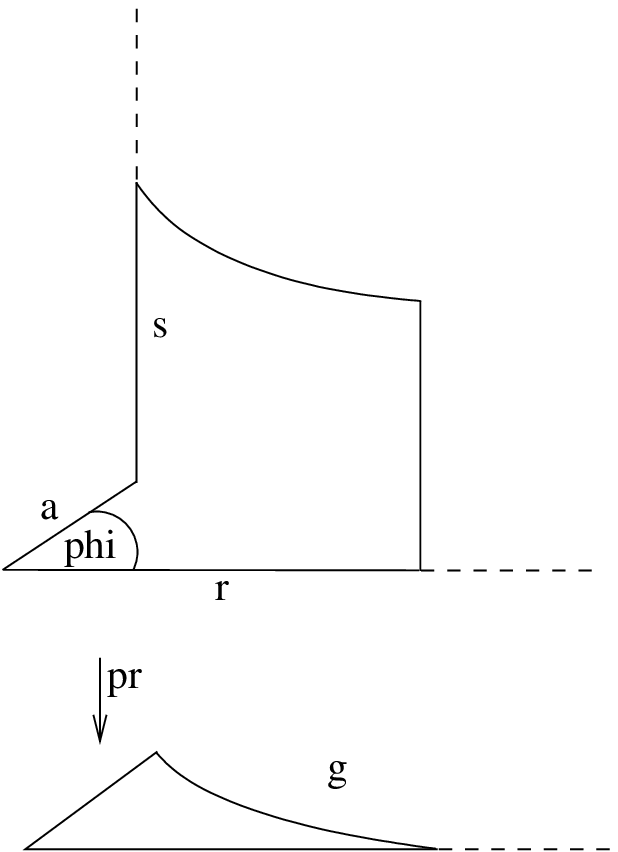}
\qquad\qquad
\includegraphics[width=5cm]{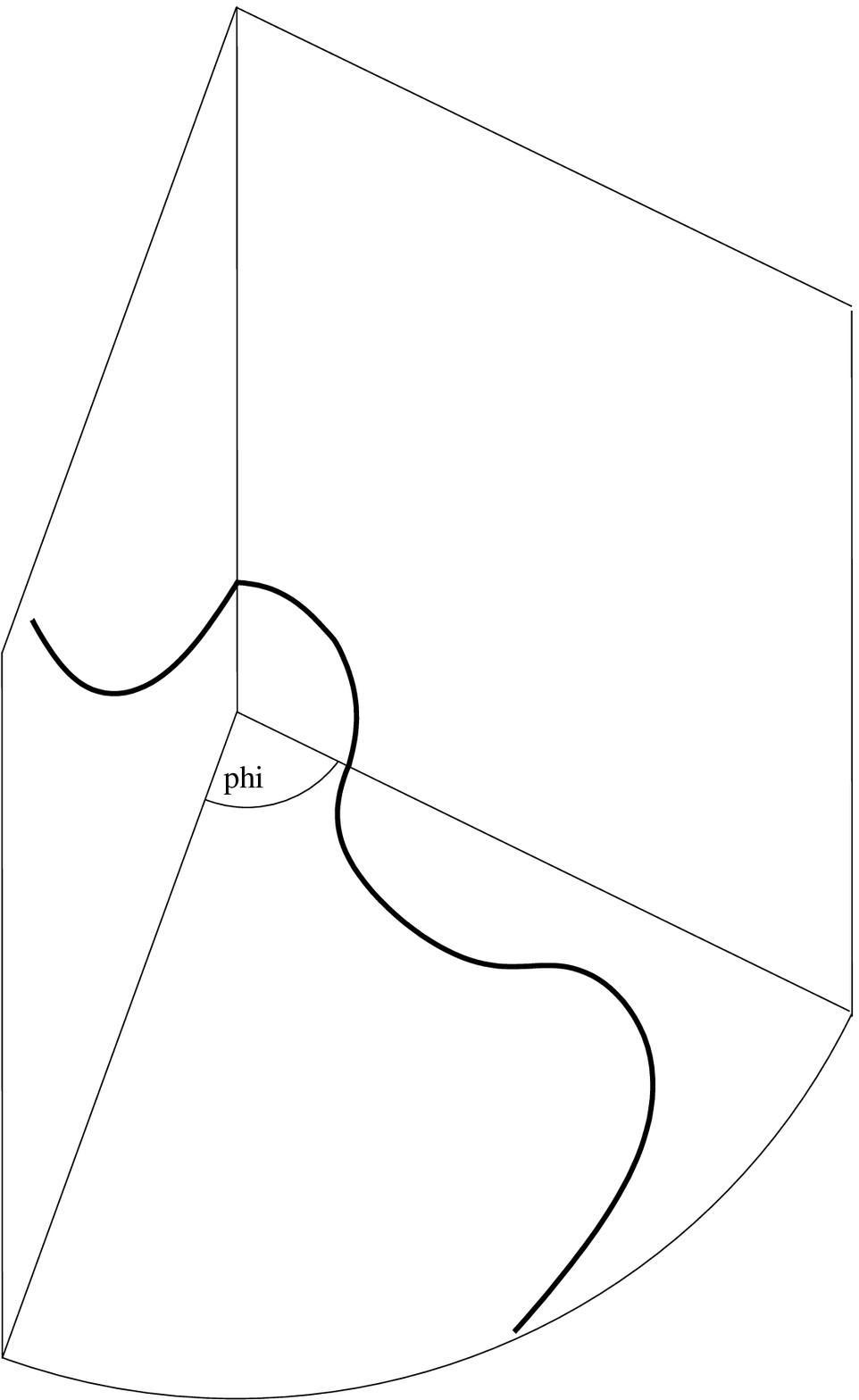}
\caption{Left: The boundary of the minimal disc in $E(\kappa+4H^2,H)$ and its projection, here $\kappa+4H^2<0$. Right: The boundary of the desired cmc surface in $\Sigma(\kappa)\times\R$.}\label{f:knoid}\end{center}\end{figure}

The Jordan curves $\Gamma_{(r,s)}$ converge to the boundary $\Gamma$ of the desired minimal disc $M$, in the sense that $\Gamma_n\cap K_x=\Gamma\cap K_x$ for any compact neighborhood $K_x\subset E(\kappa+4H^2,H)$ for $x\in\Gamma$ and $n$ large enough.
\begin{theorem}\label{T:knoid}
Suppose $\kappa+4H^2\leq0$. Then there exist minimal surfaces $M=M(a,k)$ with boundary $\Gamma$, such that $M$ is a section projecting to $\Delta\coloneqq\lim_{r\to\infty}\Delta_r$, and it extends without branch points by Schwarz reflection about the edges of $\Gamma$.
\end{theorem}

The case $\kappa=-1$, $H=1/2$, i.e. the existence of a minimal surface with those properties in $\Nil_3$ was already proven in~\cite{plehnert2013}, the case $\kappa=-1$, $H=0$ was shown in~\cite{pyo2011}.

\begin{proof}
In order to define $\Gamma_{(r,s)}$ we consider a triangle $\Delta_r$ in the base manifold $\Sigma(\kappa+4H^2)$, it is well-defined by a hinge of lengths $a$ and $r$ enclosing an angle $\pi/k$. We lift the hinge horizontally, add a vertical geodesic of length $s$ in fiber direction to the end of the edge of length $a$ and lift the remaining edge of the geodesic triangle. By Lemma~\ref{l:vertdistances} the distance of the two endpoints is $s-2\tau\vol(\Delta_r)$, therefore with $s=r^2$ the distance is always in fiber direction as sketched in Figure~\ref{f:knoid}. We define $\Gamma_r\coloneqq\Gamma_{(r,r^2)}$ by adding the remaining vertical geodesic, it is contained in the boundary of a mean convex domain. By~\cite{MY1982} there exists an embedded minimal surface $M_r$ with boundary $\Gamma_r$. Moreover, it is an unique section $u_r$ of the line bundle $\pi\colon E(\kappa+4H^2,H)\to\Sigma(\kappa+4H^2)$ projecting to $\Delta_r$ and extends without branch points, see~\cite{plehnert2013}.

To prove that the sequence of minimal sections $M_r$ converges to a minimal surface $M$ on compact subsets, we have to show that a barrier exists. We distinguish three cases to show that the sequence is uniformly bounded on compact subsets $K\subset\Delta$:
\begin{itemize}
\item $\kappa+4H^2<0$: We claim the sequence is uniformly bounded by the Scherk-type minimal graph defined by Equation \eqref{E:Scherk} in polar coordinates $(r,s)$ on the first quadrant of the upper half-plane. For $\kappa+4H^2<0$ the limit of $\Delta_r$ is a triangle in $\Sigma\coloneqq\Sigma(\kappa+4H^2)$ with one ideal vertex in $\partial\Sigma$, let $\gamma_r$ denote the edge which closes the hinge, see Figure~\ref{f:knoid}. We have seen that we find a Scherk-type minimal graph for any geodesic $\gamma\subset\Sigma$. Let $n\in\N$ such that $K\subset\Delta_n$ and consider the Scherk-type minimal surface $S_n$ for the geodesic $\gamma_n$. Since $M_n$ and $S_n$ are both graphs we can move $S_n$ in vertical direction such that it is a barrier from above. By the maximum principle the sequence $M_r$ is uniformly bounded by $S_n$ on $K$.
\item $\kappa=0$, $H=0$: For $K\subset\Delta_m$ the sequence of minimal graphs $M_r=(x,y,u_r(x,y))\subset\R^3$ is bounded from above by a helicoid $H_m$ with horizontal axis, that depends on $a$ and $\pi/k$ only.
\item $\kappa+4H^2=0$, $H\ne0$: The manifold $E(\kappa+4H^2,H)$ is isometric to $\Nil_3$. In \cite{DH2009} Daniel and Hauswirth proved the existence of a horizontal helicoid in $\Nil_3$, which is a barrier from above, see~\cite{plehnert2013}.
\end{itemize}
Hence for all pairs $(\kappa,H)$ with $\kappa+4H^2\leq0$ the sequence $M_r$ of minimal graphs is uniformly bounded on each compact subset $K\subset\Delta$. By the maximum principle the sequence is monotone increasing on $\Delta_k$, $r\geq k$ and there exists a gradient estimate by \cite{RST2010}, hence a diagonal process yields a minimal graph $M(a,k)$ with boundary $\Gamma$. 

To see that $M(a,k)$ extends without branch points by Schwarz reflection we distinguish two cases. If $p$ is not a vertex of $\Gamma$ it is no branch point by \cite{GL1973}. If $p$ is a vertex, then the angle is of the form $\pi/n$ with $n\geq2$ and $n$ copies of $M(a,k)$ obtained by successive rotations about the appropriate edges have a barrier by construction. Moreover the boundary is continuously differentiable in $p$, hence we are in the first case and $p$ is no branch point.
\end{proof}

The minimal graph $M(a,k)$ has a simply connected sister in a product space, which generates a complete cmc surface in $\Sigma(\kappa)\times\R$ by Schwarz reflection:

\begin{theorem}
For $H\in[0,1/2]$ and $\kappa+4H^2\leq0$ there exists a family of complete surfaces $M_a$ in $\Sigma(\kappa)\times\R$ with constant mean curvature $H$, $k$ ends, one horizontal and $k$ vertical symmetry planes, $a>0$.
\end{theorem}
\begin{proof}
By Daniel's correspondence \ref{t:correspondence} the minimal graph $M(a,k)\subset E(\kappa+4H^2,H)$ has a simply connected sister surface with cmc $H$ in the product manifold $\Sigma(\kappa)\times\R$, which is locally isometric. Since $M(a,k)$ is a graph the sister surface is a (multi-)graph. Moreover since $M(a,k)$ is bounded by horizontal and vertical geodesics, the sister surface is bounded by mirror curves in vertical and horizontal planes. Hence by Schwarz reflection about those planes we get a complete cmc surface $M_a$ consisting of $4k$ fundamental pieces with the claimed symmetry planes.

In order to understand the geometry of $M_a$ we analyse the mirror curves, which are the sister curves of $\partial M(a,k)$. We consider the downward pointing normal $\nu$ and parametrize $\partial M(a,k)=c$ such that the horizontal component is followed by the vertical geodesic. Let $\eta$ denote the conormal along $c$ sucht that $(c',\eta,\nu)$ is positively oriented and $c_1$, $c_2$, $c_3$ each geodesic component, where $c_3$ is vertical, i.e. $\la c_3',\eta\ra=1$.
The curve $c_1$ corresponds to a mirror curve $\tilde{c}_1$ in a vertical plane, since the conormal along $c_1$ is vertical only in the end, $\tilde{c}_1$ is a graph and $\la \tilde{c}'_1,\tilde{\nu}\ra\to -1$ in the end. Hence, the curve comes from $+\infty$. Along the finite horizontal curve $c_2$ the normal rotates about $\pi/2$ and so does the normal along $\tilde{c}_2$, moreover it is graph as $\tilde{c}_1$. Finally we consider the geodesic in fiber direction $c_3$, since the surface is graph, the normal rotates monotone. Moreover, since we have chosen the downward pointing normal and $\la c_3',\eta\ra=1$, the twist $\alpha$ is increasing and by Equation \eqref{E:twist} the curvature of the horizontal sister curve $\tilde{c}_3$ is $\tilde{k}=2H-\alpha'<2H$. We claim that $\tilde{c}_3$ is embedded. Assume the contrary, i.e. there exist $t_0,t_1$ such that $\tilde{c}_3(t_0)=\tilde{c}_3(t_1)$. We apply the Gau\ss{}-Bonnet theorem to the loop $\tilde{c}_3\vert_{[t_0,t_1]}$ which bounds a domain $\Omega$
\[
\kappa\vol{\Omega}-\int\tilde{k}+\phi=2\pi,\]
where $\phi\in[0,\pi]$ determines the angle $\measuredangle(c_3'(t_1),c_3'(t_0))$. The geodesic curvature is $-\tilde{k}$, since a loop cannot occur with the surface normal $\tilde{\nu}$ pointing inwards $\Omega$ by $\tilde{k}<2H$. But $\tilde{k}=2H-\alpha'$ implies \[\alpha\vert_{[t_0,t_1]}=2\pi-\phi+2H l(\partial\Omega)-\kappa\vol(\Omega)>\pi,\] which is a contradiction.
\end{proof}

\section{Constant mean curvature $2k$-noids}
We construct a $2$-parameter family of surfaces with cmc $H\in\left[0, 1/2\right]$ in $\Sigma(\kappa)\times\R$ with $2k$ ends and dihedral symmetry. Each surface has $k$ vertical symmetry planes and one horizontal one, where $k\geq 2$. The idea is to solve an improper Plateau problem of disc type in $E(\kappa+4H^2,H)$, where $\kappa+4H^2\leq0$, and the disc is bounded by geodesics. Its sister disc in $\Sigma(\kappa)\times\R$ generates an MC $H$ surface by reflections about horizontal and vertical planes.

For $\kappa=-1$ and $H=1/2$, it corresponds to a minimal surface in $\Nil_3(\R)=E\left(0,1/2\right)$
; if $0<H<1/2$ the MC$H$ surface results from a minimal surface in $\widetilde{\PSL}_2(\R)$ and finally for $H=0$ the surfaces are conjugate minimal surfaces in $\Sigma(\kappa)\times\R$, $\kappa\leq 0$.

\subsection{Boundary construction}

In $\Sigma(\kappa)\times\R$ the desired boundary is not connected. It consists of two components: The first component is a curve consisting of two mirror curves, each lying in a vertical plane. The two planes form an angle $\phi=\pi/k$, $k\geq2$. The second component is a mirror curve in a horizontal plane. 
\begin{figure}\begin{center}
\psfrag{a}{$\alpha$}
\psfrag{d}{$d$}
\psfrag{pr}{$\pi$}
\psfrag{p}{$\phi$}
\psfrag{phi}{$\phi$}
\psfrag{hatp}{$\hat{p}$}
\psfrag{p1}{$p_1$}
\psfrag{p1t}{$\tilde{p}_1$}
\psfrag{1}{$p_1$}
\psfrag{2}{$p_2$}
\psfrag{3}{$p_3$}
\psfrag{4}{$p_4$}
\psfrag{5}{$p_5$}
\psfrag{6}{$p_6$}
\psfrag{7}{$p_7$}
\psfrag{8}{$\hat{p}_1$}
\includegraphics[width=6cm]{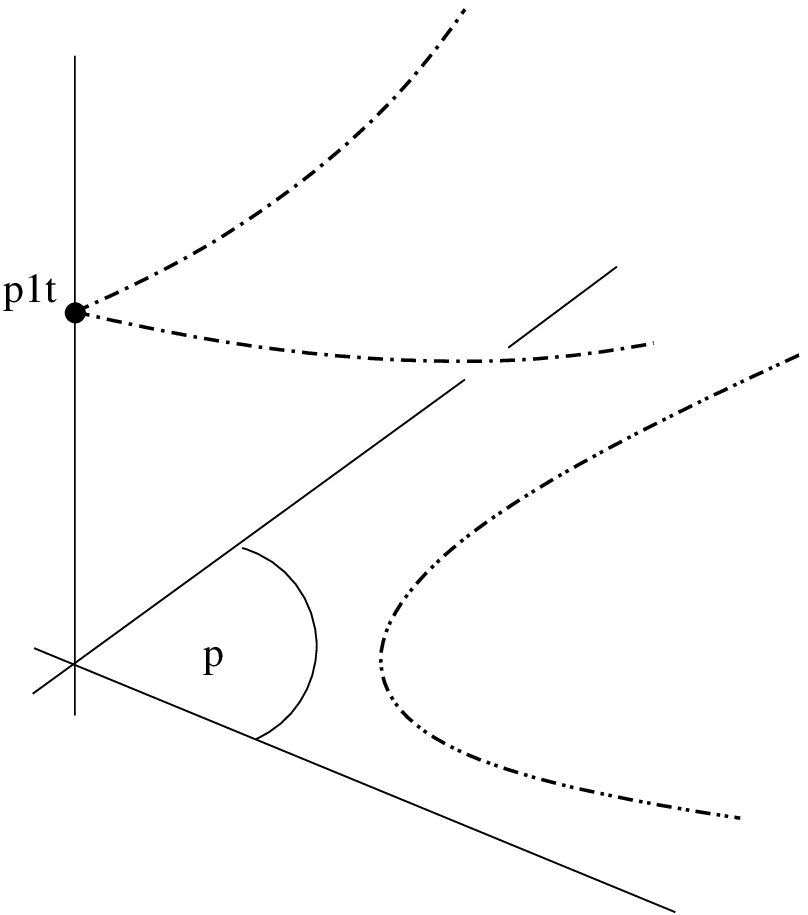}\qquad
\includegraphics[width=4.5cm]{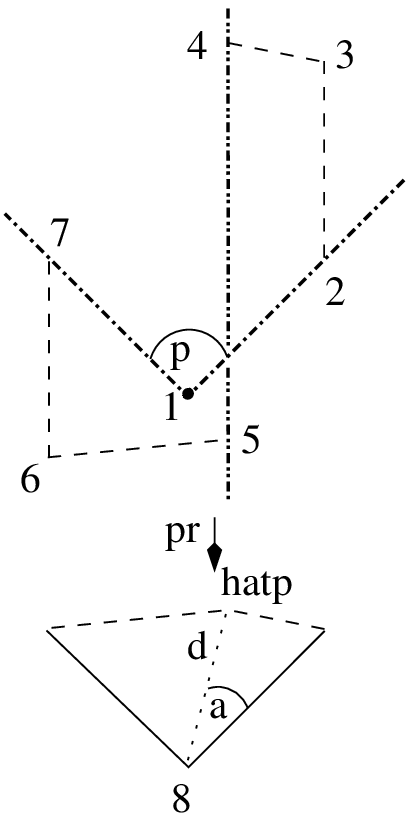}
\caption{Left: The desired boundary of the cmc surface in $\Sigma(\kappa)\times\R$. Right: The corresponding boundary of the minimal sister surface in $E(\kappa+4H^2,H)$. The single-dotted curve left corresponds to the single-dotted curve on the right.}\label{f:2knoid}\end{center}\end{figure}

In \cite{MT2011} the relation between mirror curves and their geodesic sisters is discussed. By this relation the sister surface in $E(\kappa+4H^2,H)$ is bounded by a geodesic contour $\Gamma\coloneqq\Gamma_{d,\alpha}$: The horizontal mirror curve corresponds to a vertical geodesic and the mirror curves in vertical planes are related to horizontal geodesics enclosing an angle $\pi/k$. The relative position of the vertical and horizontal components determines the $2k$-noid. 

The distance $d$ of the vertical geodesic to the vertex of the two horizontal geodesics is well-defined and realised by the length of a horizontal geodesic $\gamma$. Its length is equal to the length of its projection $\pi(\gamma)$ to $\Sigma(\kappa+4H^2)$, since it is a horizontal geodesic and the projection $\pi$ is a Riemannian fibration. The same holds for the angle $\alpha$ enclosed by $\gamma$ and one of the horizontal rays. Since the sister surfaces are isometric, it is consistent to call the $2$-parameter family of cmc surfaces which we will obtain $\widetilde{M}_{d,\alpha}$.

To construct a minimal surface that is bounded by $\Gamma$, we truncate the infinite contour $\Gamma$ and get Jordan curves $\Gamma_n$, $n>0$. To define $\Gamma_n$ we consider a geodesic quadrilateral $\Delta_n\coloneqq\Delta_n(d,\alpha)$ in $\Sigma(\kappa+4H^2)$:  Two edges of length $n$ form an angle $\pi/k$ and intersect in point $\hat{p}_1$. Furthermore, its diagonal in $\hat{p}_1$ has length $d$ and encloses an angle $\alpha\leq\phi/2$ to one side. Let $\hat{p}$ denote the endpoint of the diagonal. We consider the horizontal lift of $\Delta_n$ starting in $\hat{p}$ and going in positive direction. We label the endpoints with $\tilde{\Delta}_n(0)=p_5$ and $\tilde{\Delta}_n(l)=p_4$, by Lemma \ref{l:vertdistances} the signed vertical distance is $d(p_5,p_4)=2H\vol(\Delta_n)$ and therefore, in positive $\xi$-direction. Now we translate the horizontal edge that ends in $p_4$ in positive $\xi$-direction by $n$ and call the endpoint $p_3$. Furthermore, we translate the horizontal edge that starts in $p_5$ in $-\xi$-direction by $n$ and call the endpoint $p_6$. After the vertical translation we have
\[d(p_5,p_4)=2H\vol(\Delta_n)+2n.\]
Hence, $p_5p_4$ is in positive $\xi$-direction. We complete the Jordan curve $\Gamma_n$ by adding two vertical edges of length $n$ in $p_3$ and $p_6$ and label the intersections with the horizontal edges $p_2$ and $p_7$ respectively. See Figure \ref{f:2knoid}.

The polygon $\Gamma_n$ has six right angles and one angle $\phi=\pi/k$; the quadrilateral $\Delta_n$ is not necessarily convex for large $n$.

For $n\to\infty$ we have $\Gamma_n\to\Gamma$; this contour can be constructed by the union of quadrilaterals $$\Delta\coloneqq\Delta(d,\alpha)=\bigcup_{n>0}\Delta_n.$$ Note that $\Delta$ is non-convex.

\subsection{Plateau solutions}\label{SS:Plateau}

The idea is to consider the Plateau solutions for $\Gamma_n$ and to take their limit for $n\to\infty$. We show that there exists a domain $\Omega_n$ with mean convex boundary, such that $\Gamma_n\subset\partial\Omega_n$. A Riemannian manifold $N$ with boundary is mean convex if the boundary $\partial N$ is piecewise smooth, each smooth subsurface of $\partial N$ has non-negative mean curvature with respect to the inward normal, and there exists a Riemannian manifold $N'$ such that $N$ is isometric to a submanifold of $N'$ and each smooth subsurface $S$ of $\partial N$ extends to a smooth embedded surface $S'$ in $N'$ such that $S'\cap N=S$. We call each surface $S$ a barrier. By \cite{HS1988} the solution of the Plateau problem for $\Gamma_n$ is an embedded minimal disc. In the last step it is shown that the sequence of minimal discs has a limiting minimal disc with boundary $\Gamma$.

\begin{proposition}
The special Jordan curve $\Gamma_n\subset E(\kappa+4H^2,H)$ bounds a Plateau solution $M_n\subset E(\kappa+4H^2,H)$ for large $n\in\N$, which extends without branch points by Schwarz reflection about the edges of $\Gamma_n$.
\end{proposition}
\begin{proof}
We define the domain $\Omega_n$ with mean convex boundary as the intersection of five domains: two of them have horizontal umbrellas as boundaries, two have vertical planes as boundaries and the last domain has four fundamental pieces of the minimal $k$-noids from \ref{S:knoid} in its boundary:

\begin{enumerate}
\item Take the halfspaces above the horizontal umbrella $U_5$ in $p_5$ and below the horizontal umbrella $U_4$ in $p_4$, see Subsection \ref{SS:umbrella} for the definition. Below resp. above means in negative resp. positive $\xi$-direction. We call the intersection of the two halfspaces a {\it horizontal slab}. The umbrellas are with respect to the same fiber and therefore are parallel sections with vertical distance $d(U_5,U_4)=d(p_5,p_4)$. If $U_{4/5}\cap\Gamma_n\setminus\{\overline{p_{4/5}p_{3/6}}\}\ne\emptyset$ we redefine $\Gamma_n$ by translating $p_{4/5}$ in $\pm\xi$ direction by factor $c_{4/5}$. Since $U_{4/5}$ are sections, they are graphs above/below the horizontal geodesics $\overline{p_1p_2}$ and $\overline{p_1p_7}$. Therefore we find constants $c_{4/5}>0$ such that the new boundary curve, we call it again $\Gamma_n$ does not intersect the umbrellas except for $\overline{p_4p_3}$ and $\overline{p_5p_6}$. The horizontal slab is a barrier for $\Gamma_n$, for all $n\in\N$.

\item Furthermore, consider the vertical halfspaces defined by the horizontal arcs $\overline{p_1p_2}$ and $\overline{p_1p_7}$, such that $\Gamma_n$ lies inside.

\item The last domain is based on the minimal surface from \ref{S:knoid}: The idea is to consider a mean convex set sandwiched between two copies of a symmetric minimal $2k$-noid piece. We claim that we can orient them and choose their parameters, the necksize $a$ and the number of ends, such that they are barriers for $\Gamma_n$. Their position relative to $\Gamma_n$ is given by a rotation angle $\delta$ and the vertical distances $h_\pm$ from the horizontal umbrella $U$ in $p_1$. We call the vertical distance to $U$ {\it height}.

We take four fundamental patches $M_{d,2k}$ from Subsection \ref{S:knoid} and orient them, such that their axes coincide with $\overline{p_4 p_5}$. We require their horizontal boundaries to lie in the horizontal umbrella $U$. By rotation about $\overline{p_4 p_5}$ by an small angle $\pm \delta$  followed by vertical translations by $h_\pm$ we construct a mean convex domain $S$. Each minimal surface $S_\pm$ consists of two fundamental domains generated by Schwarz reflection about the bounded horizontal edge.

\begin{figure}
\begin{center}
\psfrag{d}{$\delta$}
\psfrag{1}{$\pi(p_1)$}
\psfrag{2}{$\pi(p_{2/3})$}
\psfrag{3}{$\pi(p_{4/5})$}
\psfrag{4}{$\pi(p_{6/7})$}
\psfrag{h}{$\mathbf{h_+}$}
\psfrag{0}{$\mathbf{0}$}
\psfrag{i}{$\mathbf{-\infty}$}
\psfrag{p}{$\mathbf{-(n+c_5)}$}
\psfrag{m}{$\mathbf{n+c_4}$}
\psfrag{k}{$\mathbf{\infty}$}
\psfrag{j}{$\mathbf{h_-}$}
\includegraphics[width=6cm]{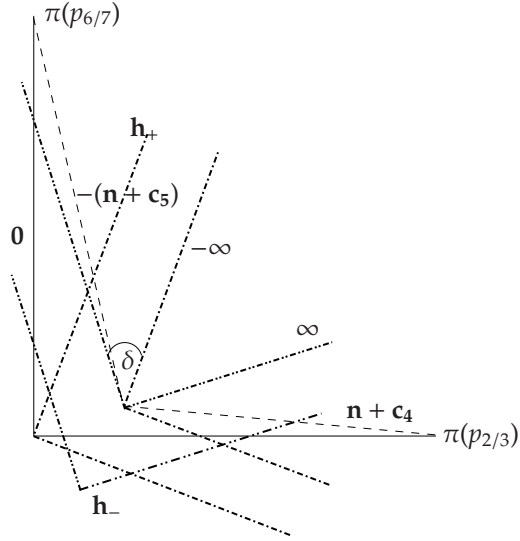}\caption{The projection of the three boundaries of $M_n$, $S_+$ (single-dotted) and $S_-$ (double-dotted) to $\Sigma(\kappa+4H^2)$. The heights of the horizontal geodesics are printed in bold letters. Here $k=2$ and $\kappa+4H^2=0$.}
\end{center}\end{figure}

We define $S_+$ first: We choose an orientation such that the projection of its horizontal edge of length $d$ coincides with the diagonal of the quadrilateral in the projection. Afterwards we translate the surface in $\xi$-direction by $h_+$. There exists $N\in\N$ such that for all $n\geq N$ the surface $S_+$ does not intersect $\Gamma_n$ for all $h_+>0$, because $S_+$ is graph above the projection of the horizontal edges of $\Gamma_n$. In the projection the horizontal hinge of $\Gamma_n$ encloses an angle $\delta\coloneqq\phi/2-\alpha>0$ with the horizontal hinge of $\partial S_+$.

To define $S_-$ we rotate the other copy of two fundamental patches $M_{d,2k}$ from Subsection \ref{S:knoid} about $\overline{p_4p_5}$ such that in the projection the two edges of length $d$ enclose an angle $\delta$. Afterwards we translate the surface by $h_-$ in $-\xi$-direction. The surface $S_-$ is a graph below a bounded component of $\overline{p_1p_2}$, therefore exists $h_->0$ such that $S_-\cap \overline{p_1p_2}=\emptyset$. Furthermore, it is a graph below $\overline{p_1p_7}$, where the surface lies below the horizontal umbrella at height $h_-$. The other edges of $\Gamma_n$ are uncritical for $n$ large enough.

\begin{figure}
\begin{center}
\psfrag{d}{$\delta$}
\psfrag{+}{$\beta_+$}
\psfrag{-}{$\beta_-$}
\includegraphics[width=6cm]{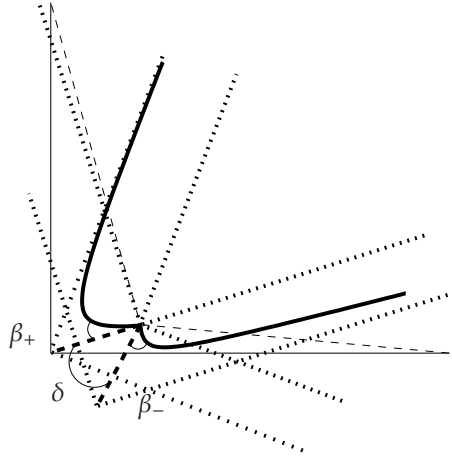}\caption{Level curves: The solid lines sketch the intersection of the horizontal umbrella $U$ with $M_n$ and the symmetric $2k$-noids (bold). The dashed lines indicate the remaining boundaries. Here $k=2$.}\label{f:level0}
\end{center}
\end{figure}

It remains to show that for every $p\in p_4 p_5$ the opening angle $\psi$ of the tangent cone $T_p C$ of $S$ is less than $\pi$. This is clear for $p$ at height $\abs{h}> \max\{h_\pm\}$. The angle has its maximum $\psi(h_+,h_-)$ in height $(h_+-h_-)/2$, it depends on $h_+$ and $h_-$ and is bounded by $\psi_{\sup}\coloneqq 2(\pi-\phi-\epsilon)+\delta$ where $\epsilon\geq0$ denotes the defect depending on $\Sigma(\kappa+4H^2)$. We consider the level curves of $S_+$ and $S_-$ in height $h$. The level curves define angles $\beta_+(h)$ and $\beta_-(h)$ given in the projection by the angle of the projected conormal in height $h$ of $S_\pm$ and the edge of length $d$ of the corresponding surface. We know $0\leq\beta_\pm<\pi-\pi/(2k)-\epsilon$. Therefore, we have \[\psi(h_+,h_-)=\beta_+\left(\frac{h_+-h_-}{2}\right)+\beta_-\left(\frac{h_+-h_-}{2}\right)+\delta.\] But $h_+$ was chosen independently of $n$, $\delta$ and $h_-$, moreover for $h_+\to 0$ we have $\beta_+\left((h_+-h_-)/2\right)\to0$. Hence we conclude 
\[\psi(h_+,h_-)\to\beta_-\left(\frac{h_+-h_-}{2}\right)+\delta<\pi-\epsilon\leq\pi.\]

We summarise: $\Gamma_n$ lies in-between two copies of a cmc $2k$-noid, i.e. there exists $N\in\N$ such that $(S_+\cup S_-)\cap \Gamma_n=p_4 p_5$ for $n\geq N$.

We complete the last barrier by subsets of the two horizontal umbrellas at heights $h_\pm$ given by the edges of the symmetric $2k$-noids; it defines the boundary of the halfspace $S$.
\end{enumerate}

We define the mean convex domain $\Omega_n$ as the intersection of the five halfspaces. Since $\Gamma_n$ lies in the boundary of a mean convex domain, the existence of an embedded minimal surface $M_n$ of disc-type with boundary $\Gamma_n$ follows from \cite{HS1988}.
\end{proof}

\begin{remark}\label{r:vertdis}
The definition of $\Omega_n$ would be more direct if we could define it as the intersection of halfspaces. A vertical plane/horizontal umbrella separates $E(\kappa+4H^2,H)$ into two connected components, but two fundamental patches of a minimal $k$-noid from Subsection \ref{S:knoid} do not separate $E(\kappa+4H^2,H)$ in two connected components. This is because of the normal turning along the vertical geodesic. To get several connected components, we have to use $S_+\cup S_-$, but their boundary would not be smooth anymore.
\end{remark}

\begin{lemma}
The Plateau solution $M_n$ is a section over a simply connected domain $\Delta_n$ enclosed by $\partial\Delta_n\coloneqq\pi(\Gamma_n)$ and unique among all Plateau solutions with the prescribed boundary values for each $n\in\N$.
\end{lemma}
\begin{proof}
We show that $M_n$ does not have any vertical tangent planes. Then Lemma \ref{l:section} below implies that $M_n$ is a section.

Suppose there exists a vertical plane $V$ that is tangent to $M_n$ at some $p\in M_n$. We consider the intersection $V\cap\overline{M_n}$: Since $M_n$ and $V$ are both minimal but not identical, their intersection $M_n\cap V$ is a union of analytic curves ending on $\partial M_n=\Gamma_n$. At $p$ at least two of them meet. Assume two curves crossing at $p$ extend to a loop $\gamma\subset M_n\cap V$, then the precompact component of $M_n\setminus\gamma$ would coincide with $V$. By the maximum principle then $M_n\equiv V$, which is impossible since $\Gamma_n\nsubseteq V$. Therefore, $M_n\cap V$ cannot contain a loop and the analytic curves have at least four endpoints on $\Gamma_n$. 

Let us consider the intersection of the vertical plane with the mean convex domain $\Omega_n\cap V$. It might consist of more than one connected component. For the connected component $V_p$ containing $p$, we know that $\Gamma_n\cap V_p$ has at most two connected components. So at least two endpoints of $M_n\cap V$ form a loop in $V\cup\overline{M}_n$, which by assumption is excluded. Therefore, we get a contradiction, i.e. there exists no vertical tangent plane. 

Since the minimal surface $M_n$ is a section $s\colon\Delta_n\to E$, it is the solution of an elliptic partial differential equation. The uniqueness follows from the maximum principle for elliptic partial differential equations.
\end{proof}

\begin{lemma}\label{l:section}
Let $\Delta\coloneqq\pi(M)$ be a compact disc. If $M$ does not have any vertical tangent planes, then $M$ is a section over $\Delta$.
\end{lemma}
\begin{proof}
Since there is no vertical tangent plane we have $\diff \pi v\ne0$ for all $v\in \T _p M$ and any $p\in M$. By the inverse mapping theorem there exists a neighbourhood $U_p\subset\Delta$ and a continuous map $s_p\colon U_p\to E$ such that $\pi(s_p(x))=x$ for all $x\in U_p$. Moreover, there exists a finite covering $\{U_{p_n}\}_n$ of $\Delta$ and the inverse maps coincide for $U_{p_k}\cap U_{p_l}\ne\emptyset$. Therefore, we get a continuous map $s:\Delta\to E$ such that $\pi(s(x))=x$ for all $x\in \Delta$.
\end{proof}

Now we can take the limit $n\to\infty$:
\begin{theorem}
There exists a minimal surface $M_\infty \subset E(\kappa+4H^2,H),\,\kappa+4H^2\leq0$ which is a section over $\Delta$ and extends without branch points by Schwarz reflection across its edges.
\end{theorem}
\begin{proof}
Since $M_n$ is a sequence of sections, it is monotone increasing on $\Delta_k,\, n\geq k$. Moreover by the gradient estimate it is sufficient to prove that the sequence is uniformly bounded on each compact $K\subset\Delta$. We modify the proof of Theorem \ref{T:knoid}: For $K\subset\Delta$ we consider $k\in\N$ such that $K\subset\Delta_k$ and two fundamental pieces $M_{\pm}$ of a minimal $k$-noid from above, one with the end going to infinity and the other going to minus infinity. As before we may orientate them such that the positive end of $M_+$ lies above the boundary component of $M_k$ with value $k$ and the negative end of $M_K$ lies below the boundary component of $M_k$ with value $-k$. We continue $M_\pm$ with a horizontal umbrella such that each is a minimal section well-defined on $\Delta_k$. By the maximum principle there is no point of contact if we consider the sequence $M_n$ for $n\geq k$. After diagonalization we obtain a minimal surface $M_\infty$ which is a section over $\Delta$. As is the proof of Theorem \ref{T:knoid} it extends by construction without branch points by Schwarz reflection across its edges.
\end{proof}

The minimal surface $M_\infty$ is a fundamental piece of a $2k$-noid; we reflect its sister surface to construct a cmc surface in $\Sigma(\kappa)\times\R$: 

\begin{theorem}
For $H\in\left[0, 1/2\right]$ and $k\geq2$ there exists a two-parameter family \[\left\{\widetilde{M}_{d,\alpha}\colon d>0, 0<\alpha\leq\pi/(2k)\right\}\] of constant mean curvature $H$ surfaces in $\Sigma(\kappa)\times\R$, $\kappa\leq0$, such that:

\noindent
$\bullet\, \widetilde{M}_{d,\alpha}$ has $k$ vertical mirror planes enclosing an $\pi/k$-angle,

\noindent
$\bullet\, \widetilde{M}_{d,\alpha}$ has one horizontal mirror plane and

\noindent
$\bullet\,$for $\alpha=\pi/(2k)$ the surface $\widetilde{M}_{d,\alpha}$ is symmetric and coincides with the surface from Subsection \ref{S:knoid}.
\end{theorem}
\begin{proof}
By Daniel's correspondence (\cite{daniel2007}, see Theorem \ref{t:correspondence} above) the fundamental piece $M_\infty$ has a sister surface $\widetilde{M}_\infty$ with constant mean curvature $H$ in $\Sigma(\kappa)\times\R$. By construction, $\widetilde{M}_\infty$ has three curves in mirror planes: one in a horizontal and two in vertical planes; the two vertical mirror planes enclose an angle $\pi/k$. Schwarz reflection about those planes extends the surface to a complete MC $H$ surface $\widetilde{M}_{d,\alpha}$ with $2k$ ends. The MC $H$-surface $\widetilde{M}_{d,\alpha}$ consists of $4k$ fundamental pieces $\widetilde{M}_\infty$.
\end{proof}

\bibliographystyle{amsalpha}
\bibliography{/Users/juliaddg/Documents/bibliography}

\end{document}